\documentclass[12pt]{article}
\usepackage{amssymb}
\usepackage{amsmath}
\usepackage{amsthm}
\usepackage{enumitem}
\usepackage{geometry}		
\usepackage{graphicx}

\DeclareFontFamily{OT1}{pzc}{}
\DeclareFontShape{OT1}{pzc}{m}{it}{<-> s * [1.10] pzcmi7t}{}
\DeclareMathAlphabet{\mathpzc}{OT1}{pzc}{m}{it}

\let\originalleft\left
\let\originalright\right
\renewcommand{\left}{\mathopen{}\mathclose\bgroup\originalleft}
\renewcommand{\right}{\aftergroup\egroup\originalright}

\begin{document}

\def\b0{{\bf 0}}
\def\ee{\varepsilon}
\def\co{\mathpzc{o}}
\def\cS{\mathcal{S}}
\def\rD{{\rm D}}
\def\rT{{\rm T}}
\def\sphMeas{\textnormal{sph-meas}}

\newcommand{\removableFootnote}[1]{}

\newtheorem{theorem}{Theorem}[section]
\newtheorem{corollary}[theorem]{Corollary}
\newtheorem{lemma}[theorem]{Lemma}
\newtheorem{proposition}[theorem]{Proposition}
\newtheorem{conjecture}[theorem]{Conjecture}

\theoremstyle{definition}
\newtheorem{definition}{Definition}[section]
\newtheorem{example}[definition]{Example}

\theoremstyle{remark}
\newtheorem{remark}{Remark}[section]



\title{
Chaotic attractors from border-collision bifurcations: stable border fixed points and determinant-based Lyapunov exponent bounds.
}
\author{
D.J.W.~Simpson\\\\
School of Fundamental Sciences\\
Massey University\\
Palmerston North\\
New Zealand
}
\maketitle



\begin{abstract}

The collision of a fixed point with a switching manifold (or border) in a piecewise-smooth map can create many different types of invariant sets. This paper explores two techniques that, combined, establish a chaotic attractor is created in a border-collision bifurcation in $\mathbb{R}^d$ $(d \ge 1)$. First, asymptotic stability of the fixed point at the bifurcation is characterised and shown to imply a local attractor is created. Second, a lower bound on the maximal Lyapunov exponent is obtained from the determinants of the one-sided Jacobian matrices associated with the fixed point. Special care is taken to accommodate points whose forward orbits intersect the switching manifold as such intersections can have a stabilising effect. The results are applied to the two-dimensional border-collision normal form focusing on parameter values for which the map is piecewise area-expanding.

\end{abstract}

\section{Introduction}
\label{sec:intro}
\setcounter{equation}{0}

A map on $\mathbb{R}^d$ $(d \ge 1)$ is a discrete-time dynamical system
\begin{equation}
x \mapsto f(x),
\label{eq:fGeneral}
\end{equation}
where $f : \mathbb{R}^d \to \mathbb{R}^d$.
Given an initial state $x \in \mathbb{R}^d$,
the $n^{\rm th}$ iterate $f^n(x)$ represents the state of the system after $n$ time steps.
We are usually most interested in the long-time (large $n$) behaviour of typical $x \in \mathbb{R}^d$
which is governed by the {\em attractors} of $f$.
By attractors we mean topological attractors:
invariant sets that attract a neighbourhood of initial points
and satisfy some indivisibility property (e.g.~contain a dense orbit) \cite{El08,GuHo86}.
This paper concerns attractors of maps that are piecewise-smooth.
Such maps have different functional forms in different subsets of $\mathbb{R}^d$
and arise naturally when modelling physical phenomena with abrupt events,
such as mechanical systems with impacts \cite{DiBu08},
control systems with relays \cite{ZhMo08},
and social processes with decisions \cite{PuSu06}.

As the parameters of a map are varied, attractors (and invariant sets more generally)
undergo fundamental changes at {\em bifurcations} \cite{Ku04}.
The simplest type of bifurcation novel to piecewise-smooth maps is a border-collision bifurcation (BCB).
A BCB occurs when a fixed point collides with a switching manifold (where the functional form of the map changes),
see Fig.~\ref{fig:qqGl16e}.
This paper concerns BCBs for maps $f$ that, at least locally, are continuous across a smooth switching manifold
and away from the switching manifold $D f$ is continuous and bounded, see \cite{Si16} for a review.

\begin{figure}[b!]
\begin{center}
\setlength{\unitlength}{1cm}
\begin{picture}(13,4.45)
\put(0,0){\includegraphics[height=4cm]{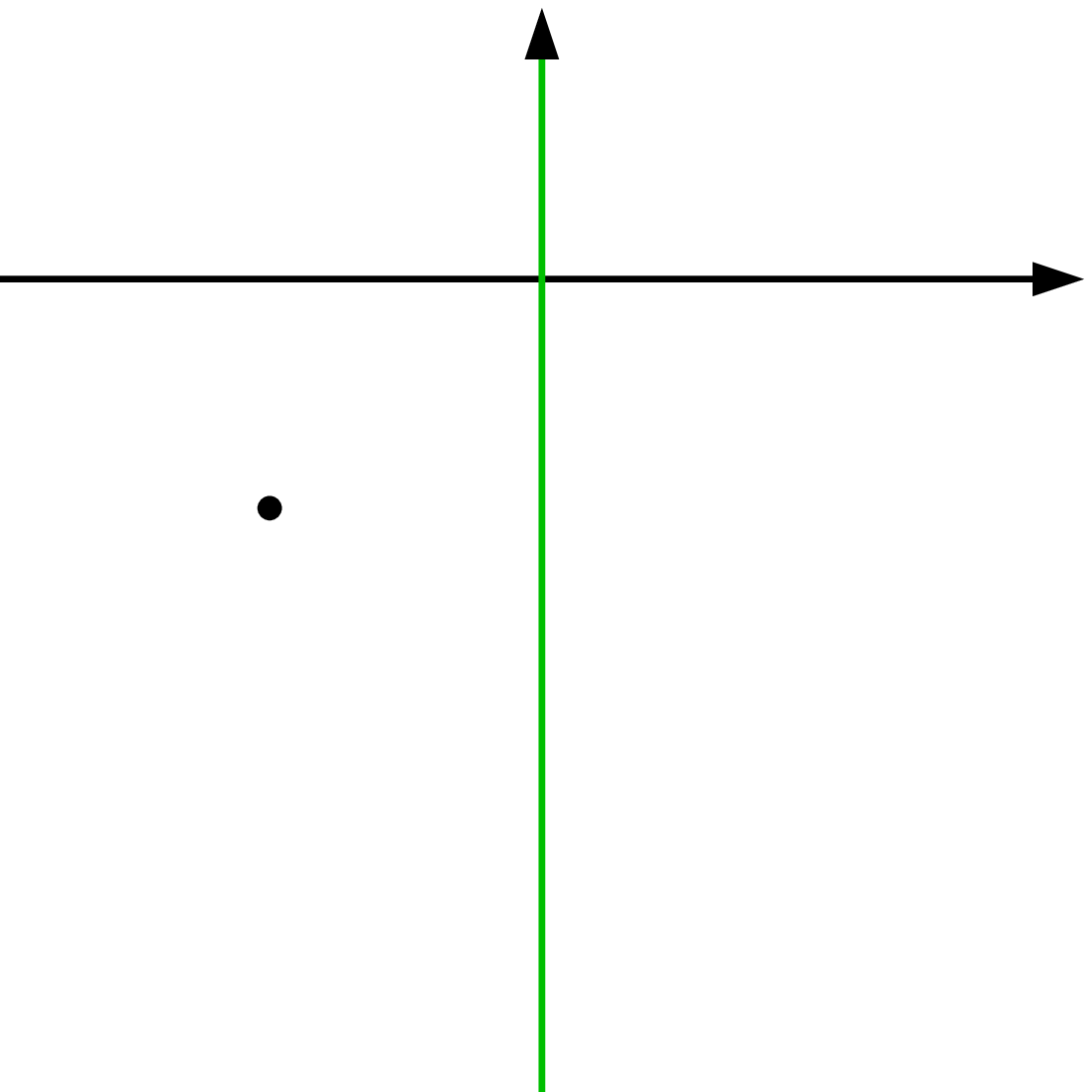}}
\put(4.5,0){\includegraphics[height=4cm]{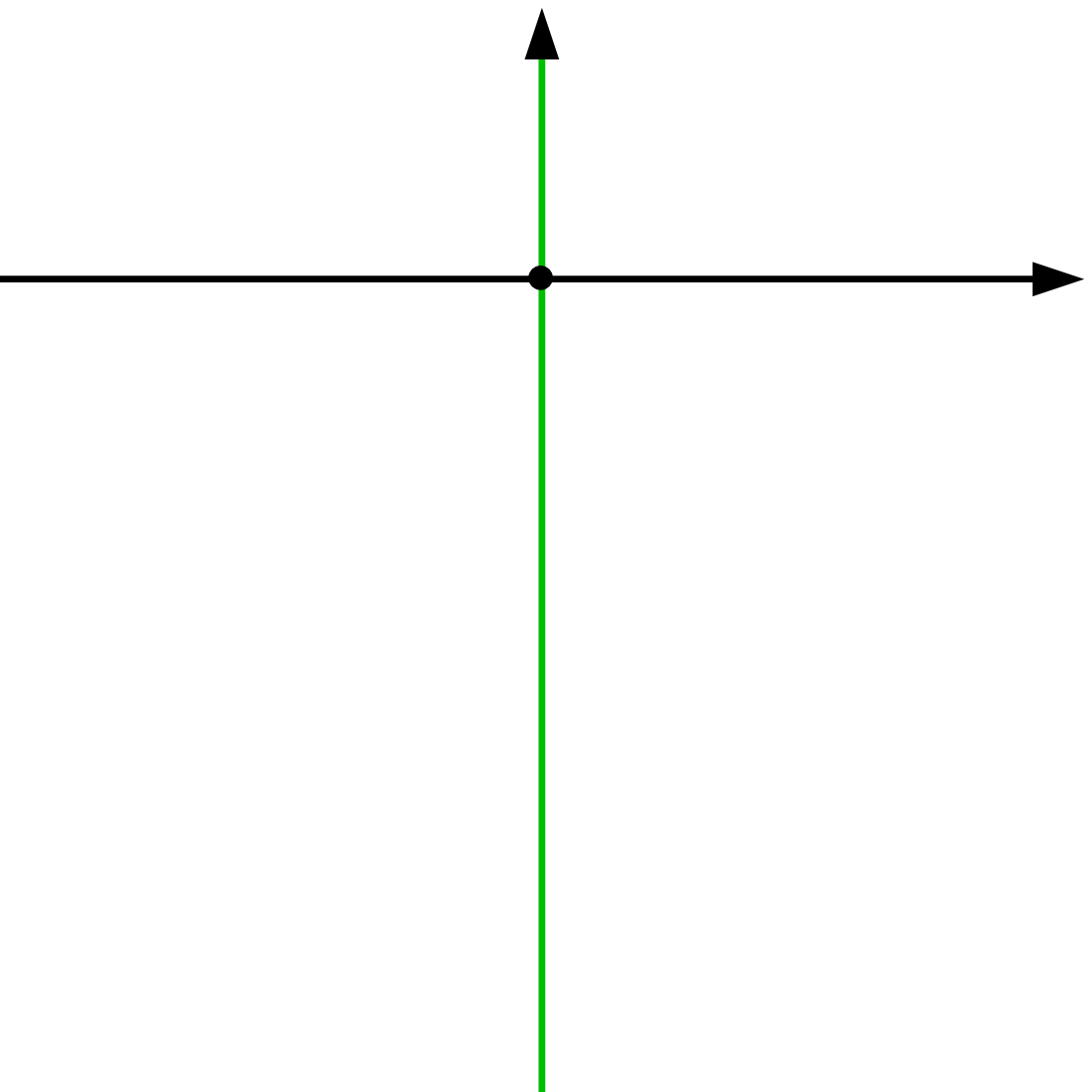}}
\put(9,0){\includegraphics[height=4cm]{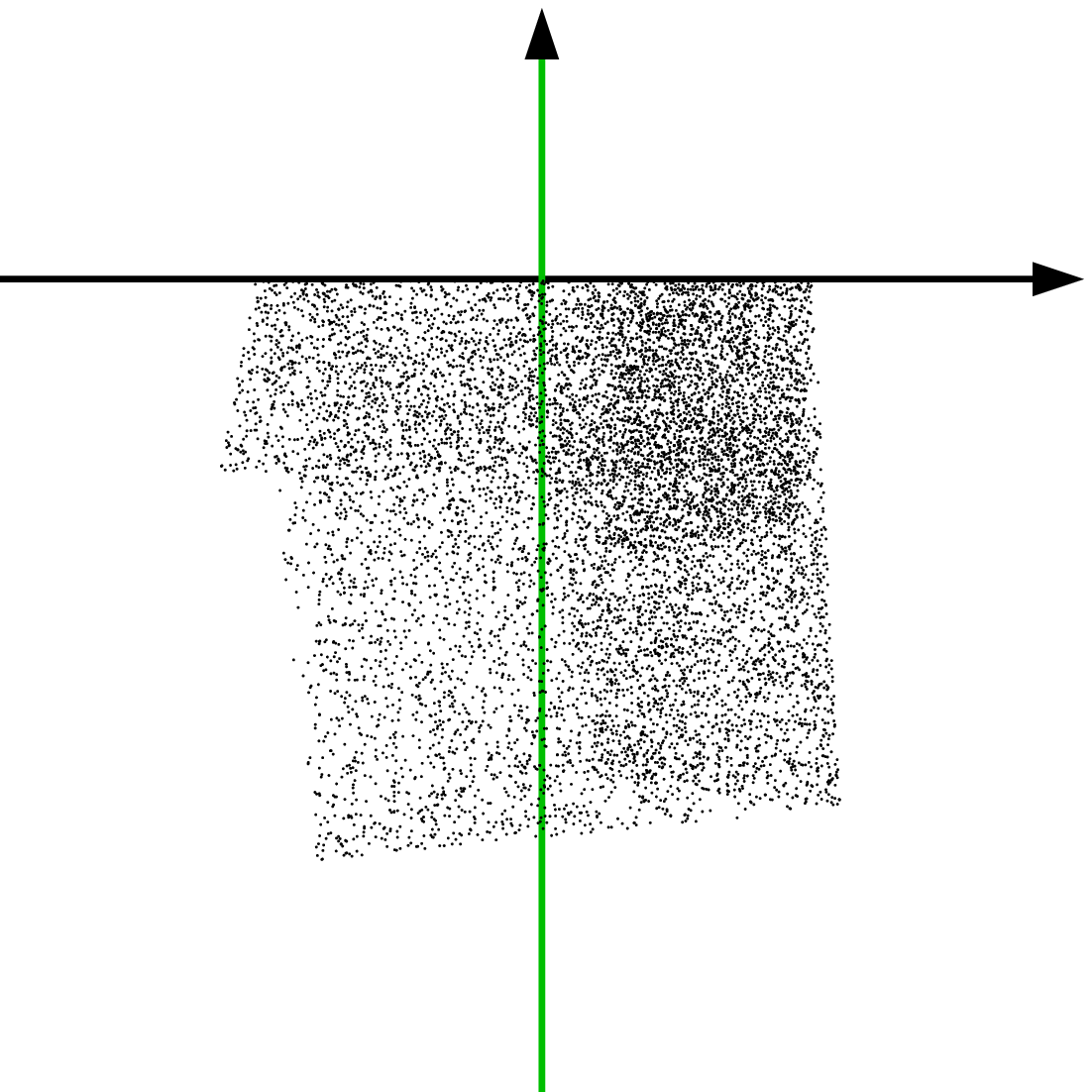}}
\put(3.43,3.13){\footnotesize $x_1$}
\put(2.1,3.75){\footnotesize $x_2$}
\put(1.56,4.15){\footnotesize $\mu < 0$}
\put(7.93,3.13){\footnotesize $x_1$}
\put(6.6,3.75){\footnotesize $x_2$}
\put(6.06,4.15){\footnotesize $\mu = 0$}
\put(12.43,3.13){\footnotesize $x_1$}
\put(11.1,3.75){\footnotesize $x_2$}
\put(10.56,4.15){\footnotesize $\mu > 0$}
\end{picture}
\caption{
\label{fig:qqGl16e}
Phase portraits of the two-dimensional border-collision normal form \eqref{eq:bcnf}
with $(\tau_L,\delta_L,\tau_R,\delta_R) = (0.05,-0.85,0.1,1.95)$, as in \cite{Gl16e}.
The map is continuous but non-differentiable on the switching manifold $x_1 = 0$ (green).
As the value of the parameter $\mu$ is increased
a stable fixed point collides with the switching manifold
when $\mu = 0$ and a chaotic attractor is created.
} 
\end{center}
\end{figure}

BCBs can create periodic, quasiperiodic, and chaotic attractors.
The internal dynamics of DC/DC power converters, for example, can exhibit a
sudden transition to chaos via a BCB \cite{DiGa98}.
Naturally one would like to know what attractors are created in a given BCB,
but in general this is an extremely difficult problem.
There are infinitely many possibilities
even simply for the number of stable periodic solutions \cite{Si14}.
For this reason it seems that in order for us to usefully expand upon the state-of-the-art theory for BCBs,
rather than pursue a detailed classification we should search for conditions
for certain types of behaviour to occur \cite{Gl17c},
and this philosophy directs the present work.

This paper addresses the following problem:
{\em under what conditions is a chaotic attractor created in a BCB}?
Arguably the simplest geometric tool that can help us here is a trapping region
(a compact set $\Omega \subset \mathbb{R}^d$ for which $f(\Omega)$ is contained in the interior of $\Omega$)
as such regions necessarily contain attractors.
To then demonstrate chaos a wide variety of techniques have been employed.
For two-dimensional maps, Misiurewicz \cite{Mi80} used the expansion and folding behaviour of unstable manifolds
to prove there exists an attractor on which the map is transitive (and hence chaotic in this sense).
Collet and Levy \cite{CoLe84} subsequently constructed an invariant measure on this attractor and
derived its basic ergodic properties.
Recently Misiurewicz's result was generalised to the 
well-known robust chaos parameter regime of \cite{BaYo98}
via the construction of a forward invariant expanding cone in tangent space \cite{GlSi19}.
Glendinning \cite{Gl17} used a theorem of Young \cite{Yo85}
to show that in part of this regime the map has an SRB measure.
Also in \cite{Gl14,Gl15b}, Glendinning applied a dimensionality result for piecewise-expanding maps \cite{Bu99,Ts01}
to show that BCBs can create chaotic attractors of dimension equal to that of the phase space,
as in Fig.~\ref{fig:qqGl16e}.

In this paper chaotic attractors are obtained through two alternate techniques.
First the existence of a local attractor is established from the stability of the fixed point at the BCB.
Then a lower bound on the maximal Lyapunov exponent is given in terms of
the determinants of the associated one-sided Jacobian matrices.
These techniques may allow us to prove the existence of a chaotic attractor
in parameter regions where other methods are inconclusive.

\subsection{Main ideas and overview}
\label{sub:overview}

The remainder of this paper is organised as follows.
We first briefly summarise key notation in \S\ref{sub:notation}.
Then in \S\ref{sec:stability} we explain how the stability of a fixed point at a BCB in a map $f$ on $\mathbb{R}^d$
can be addressed by studying the stability of the origin for a piecewise-linear map.
In general the stability problem for piecewise-linear maps is remarkably complicated \cite{BlTs99}.
Following \cite{AtLa14} we show how stability can be characterised by iterating
a compact set $\Omega$ that contains the origin in its interior, Theorem \ref{th:stability}.
By choosing $\Omega$ to be a polytope this provides an effective
numerical procedure by which stability can be established.
We also prove that if the origin is asymptotically stable,
then the corresponding BCB must create a local attractor, Theorem \ref{th:attractingSet}.

To show that attractors are chaotic we use {\em Lyapunov exponents}.
As discussed in \S\ref{sub:lyapSmooth}, for a smooth map $f$ on $\mathbb{R}^d$ Lyapunov exponents are defined by
\begin{equation}
\lambda(x,v) = \lim_{n \to \infty} \frac{1}{n} \ln \left( \left\| D f^n(x) v \right\| \right),
\label{eq:lyapSmoothIntro}
\end{equation}
assuming the limit exists.
The Lyapunov exponent $\lambda(x,v)$
represents rate at which the forward orbits of $x$ and $x + \delta v$ diverge, in the $\delta \to 0$ limit.
On an attractor we may have $\lambda(x,v) > 0$ (almost everywhere)
in which case the dynamics is locally expanding---this is the essence of chaos\removableFootnote{
Essentially I'm using $\lambda > 0$ as my definition of chaos.
In \cite{AbBi04} it is shown that for maps on $[0,1]$ that have a Markov partition,
a positive Lyapunov exponent implies sensitive dependence on initial conditions.
}.

For a piecewise-smooth map $f$,
$\rD f$ is undefined at points on switching manifolds.
In this case the limiting rate of divergence is instead given by
\begin{equation}
\lambda(x,v) = \lim_{n \to \infty} \frac{1}{n} \ln \left( \left\| C_n(x,v) v \right\| \right),
\label{eq:lyapPWSIntro}
\end{equation}
where $C_n$ is a matrix defined in \S\ref{sub:lyapPWS}.
The key difference between \eqref{eq:lyapSmoothIntro} and \eqref{eq:lyapPWSIntro}
is that $C_n$ depends on $v$, whereas $D f^n(x)$ does not.
This is because if $x$ lies on a switching manifold
then the smooth component of $f$ that we must use
depends on which side the of the switching manifold the point $x + \delta v$ lies.
Lyapunov exponents have also been described for more general classes of continuous but non-differentiable maps,
see \cite{BaSi05,BeSi12,Ki83}\removableFootnote{
In \cite{BaSi05,Ki83}, maximal Lyapunov-like exponents are defined by taking the supremum over $v$
inside the $n \to \infty$ limit.
In \cite{BeSi12}, Lyapunov-like exponents for a given $v$ are considered.
}.
In \S\ref{sub:bound} we derive a lower bound on the maximal value of \eqref{eq:lyapPWSIntro}
based on the determinants of the two one-sided Jacobian matrices associated with a BCB.
For random matrix products,
lower bounds have recently been obtained using analytic functions \cite{Po10,PrJu13}\removableFootnote{
See also \cite{StTh19} for two matrices.
}.
Then in \S\ref{sub:numerics} we demonstrate the results with the
two-dimensional border-collision normal form.
This is the map
\begin{equation}
f_\mu(x) = \begin{cases}
A_L x + b \mu, & x_1 \le 0, \\
A_R x + b \mu, & x_1 \ge 0,
\end{cases}
\label{eq:bcnf}
\end{equation}
where
\begin{align}
A_L &= \begin{bmatrix} \tau_L & 1 \\ -\delta_L & 0 \end{bmatrix}, &
A_R &= \begin{bmatrix} \tau_R & 1 \\ -\delta_R & 0 \end{bmatrix}, &
b &= \begin{bmatrix} 1 \\ 0 \end{bmatrix},
\label{eq:ALAR}
\end{align}
and $\tau_L, \delta_L, \tau_R, \delta_R \in \mathbb{R}$ \cite{NuYo92}.
Finally \S\ref{sec:conc} provides concluding remarks.

\subsection{Notation}
\label{sub:notation}

We study maps on $\mathbb{R}^d$ ($d \ge 1$) with the Euclidean norm $\| x \| = \sqrt{x_1^2 + \cdots + x_d^2}$.
The origin or zero vector is denoted $\b0$.
Open balls are written as $B_\ee(x) = \left\{ y \in \mathbb{R}^d \,\middle|\, \| x-y \| < \ee \right\}$.
The tangent space to a point $x$ is denoted $T \mathbb{R}^d$.
Here we omit the dependence on $x$ because $T \mathbb{R}^d$ is isomorphic to $\mathbb{R}^d$ for all $x$;
indeed we simply treat tangent vectors as elements of $\mathbb{R}^d$.
The unit sphere is denoted $\mathbb{S}^{d-1} = \left\{ v \in T \mathbb{R}^d \,\middle|\, \| v \| = 1 \right\}$.

If $X$ is a subset of $Y$ we write $X \subset Y$.
We write ${\rm int}(X)$ and ${\rm cl}(X)$ for the interior and closure of $X$.
The $d$-dimensional Lebesgue measure of a set $\Omega \subset \mathbb{R}^d$ is written as ${\rm meas}(\Omega)$.
For a set $\gamma \subset \mathbb{S}^{d-1}$, we use the {\em spherical measure} which is defined by
\begin{equation}
\sphMeas(\gamma) = \frac{{\rm meas} \left( \left\{ \alpha v \,\middle|\, v \in \gamma,\, 0 \le \alpha \le 1 \right\} \right)}
{{\rm meas} \left( B_1(\b0) \right)} \;.
\label{eq:sphMeas}
\end{equation}
We use $\co$ for little-o notation \cite{De81}.
In particular a function $\phi : \mathbb{R} \to \mathbb{R}$ is said to be $\co \left( \delta \right)$
if $\lim_{\delta \to 0} \frac{\phi(\delta)}{\delta} = 0$.

\section{The stability of a fixed point at a border-collision bifurcation}
\label{sec:stability}
\setcounter{equation}{0}

We first recall some basic definitions.
A fixed point $x^*$ of a continuous map $f$ on $\mathbb{R}^d$ is {\em Lyapunov stable}
if for all $\ee > 0$ there exists $\delta > 0$ such that
$f^i(x) \in B_\ee(x^*)$ for all $x \in B_\delta(x^*)$ and all $i \ge 0$.
The point $x^*$ is {\em asymptotically stable} if it is Lyapunov stable
and there exists $\delta > 0$ such that $f^i(x) \to x^*$ as $i \to \infty$ for all $x \in B_\delta(x^*)$.

\subsection{Reduction to a piecewise-linear map}
\label{sub:pwl}

Let $f_\mu$ be a continuous piecewise-$C^1$ map on $\mathbb{R}^d$
with parameter $\mu \in \mathbb{R}$.
Suppose $f_\mu$ undergoes a BCB at a point on a single smooth switching manifold $\Sigma$
that, locally, divides phase space into two regions.
Only two pieces of $f_\mu$ influence the local dynamics
and we assume coordinates can be chosen so that $\Sigma$ is the plane $x_1 = 0$, that is
\begin{equation}
f_\mu(x) = \begin{cases}
f_{L,\mu}(x), & x_1 \le 0, \\
f_{R,\mu}(x), & x_1 \ge 0,
\end{cases}
\label{eq:fPWS}
\end{equation}
where $f_{L,\mu}$ and $f_{R,\mu}$ are $C^1$.

We further assume the BCB occurs at $x = \b0$ when $\mu = 0$
and that $A_L = \rD f_{L,0}(\b0)$ and $A_R = \rD f_{R,0}(\b0)$ are well-defined.
By assumption $f$ is continuous on $\Sigma$
thus $A_R$ can only differ from $A_L$ in its first column.
We use these matrices to form the piecewise-linear map
\begin{equation}
g(x) = \begin{cases}
A_L x, & x_1 \le 0, \\
A_R x, & x_1 \ge 0.
\end{cases}
\label{eq:g}
\end{equation}
This map approximates $f_0$ near $x = \b0$,
specifically $f_0(x) = g(x) + \co(\|x\|)$,
and we have the following result.

\begin{lemma}
Consider \eqref{eq:fPWS} with $\mu = 0$ and suppose $x = \b0$ is a fixed point.
If $\b0$ is an asymptotically stable fixed point of its piecewise-linear approximation \eqref{eq:g}
then $\b0$ is an asymptotically stable fixed point of \eqref{eq:fPWS}.
\label{le:approx}
\end{lemma}

Lemma \ref{le:approx} justifies our subsequent study of $g$
(although its converse is not true in general).
Lemma \ref{le:approx} is an immediate consequence of Theorem 1 of \cite{Si16d}
(also a similar result is achieved in \cite{DeSc73})
so we do not provide a proof.

\subsection{Stability for piecewise-linear maps}
\label{sub:stabilityPWL}

The problem of the stability of $\b0$ for the piecewise-linear map \eqref{eq:g}
is a special case of a common and well-studied problem in control theory.
Many control systems are well modelled by piecewise-linear maps of the form $x(i+1) \mapsto A_{\sigma(i)} x(i)$,
where $A_1,\ldots,A_N$ are real-valued $d \times d$ matrices
and $\sigma(i) \in \{ 1,\ldots,N \}$ indicates which matrix is applied at the $i^{\rm th}$ time step.
The function $\sigma$ may be state-dependent (as in our case),
a pre-determined signal (discretised wave-form),
or determined by a particular control strategy.
One would like to know if $\b0$ is asymptotically stable for all possible $\sigma$\removableFootnote{
If each $A_j$ has eigenvalues with modulus less than $1$,
then $\b0$ may be unstable (dangerous BCBs provide an example)
but only if the dwell time is too small (discussed in the introduction of \cite{GeCo06}).
},
or, more weakly, if there exists $\sigma$ for which $\b0$ is asymptotically stable \cite{FoVa12,GeCo06,LiAn09}\removableFootnote{
Also there seems to be a lot of recent work done on accommodating uncertainties and time-delay.
}.

Broadly speaking the mathematical tool of choice
to address these problems is the {\em Lyapunov function}\removableFootnote{
For Lyapunov functions for PWL ODEs (that are possibly discontinuous) see \cite{JoRa98,Fe02}.
}.
We instead follow \cite{AtLa14}
and consider the images of a compact region $\Omega$ for which $\b0 \in {\rm int}(\Omega)$.
As discussed below this approach is well-suited to numerical implementation, allows any such $\Omega$,
and, unlike commonly used classes of Lyapunov functions,
provides an exact characterisation of the asymptotic stability of $\b0$.

\subsection{Stability for linearly homogeneous maps}
\label{sub:lh}

The map \eqref{eq:g} is an example of a {\em linearly homogeneous} map:
a continuous map $g$ on $\mathbb{R}^d$ for which $g(\alpha x) = \alpha g(x)$
for all $x \in \mathbb{R}^d$ and all $\alpha \ge 0$.
Notice that every linearly homogeneous map has $\b0$ as a fixed point.

\begin{theorem}
Let $g$ be a continuous, linearly homogeneous map on $\mathbb{R}^d$.
Let $\Omega \subset \mathbb{R}^d$ be compact
with $\b0 \in {\rm int}(\Omega)$.
The following are equivalent\removableFootnote{
I conjecture that the following statement is also equivalent
to the asymptotic stability of $\b0$:
there exists compact $\Omega \subset \mathbb{R}^d$ with $\b0 \in {\rm int}(\Omega)$
such that $f(\Omega) \subset {\rm int}(\Omega)$.
It is easy to show this condition implies asymptotic stability
following my demonstration of \eqref{it:3} $\Rightarrow$ \eqref{it:1} here.
The converse seems far more difficult,
but it's probably easy to show that the existence of such an $\Omega$
is equivalent to the existence of a strong Lyapunov function.
In this case it remains to show that asymptotic stability
implies the existence of a strong Lyapunov function
(so we need a `converse Lyapunov theorem')
but I don't know how to do this
(for ODEs such a Lyapunov function can be constructed from the flow,
see Theorem 4.8 of \cite{Me07}).
}.
\begin{enumerate}[label=\roman*),ref=\roman*,itemsep=0mm]
\item
$\b0$ is an asymptotically stable fixed point of $g$,
\label{it:1}
\item
there exists $m \ge 1$ such that $g^m(\Omega) \subset {\rm int}(\Omega)$, and
\label{it:2}
\item
there exists $m \ge 1$ such that
$g^m(\Omega) \subset {\rm int} \left( \bigcup_{i=0}^{m-1} g^i(\Omega) \right)$.
\label{it:3}
\end{enumerate}
\label{th:stability}
\end{theorem}

Theorem \ref{th:stability}, proved below, generalises a result of \cite{Si16d}
and is motivated by \cite{AtLa14} which focuses on convex subsets of $\mathbb{R}^d$.
While condition \eqref{it:2} provides a simpler characterisation of
the asymptotic stability of $\b0$ than condition \eqref{it:3},
the latter condition is more effective numerically
because the minimum value of $m$ for which condition \eqref{it:3}
is satisfied is often significantly smaller than
the minimum value of $m$ for which condition \eqref{it:2} is satisfied.

To prove Theorem \ref{th:stability} we use the following result
that was proved in \cite{Si16d}.

\begin{lemma}
Let $g$ be a continuous, linearly homogeneous map on $\mathbb{R}^d$.
Suppose there exists $\delta > 0$ such that $g^n(x) \to \b0$ as $n \to \infty$ for all $x \in B_\delta(\b0)$.
Then
\begin{enumerate}[label=\roman*),ref=\roman*,itemsep=0mm]
\item
$\b0$ is an asymptotically stable fixed point of $g$, and
\label{it:asyStable}
\item
the convergence $g^n(x) \to \b0$ is uniform on $B_\delta(\b0)$.
\label{it:uniformConvergence}
\end{enumerate}
\label{le:stability}
\end{lemma}

\begin{figure}[b!]
\begin{center}
\setlength{\unitlength}{1cm}
\begin{picture}(6,6)
\put(0,0){\includegraphics[height=6cm]{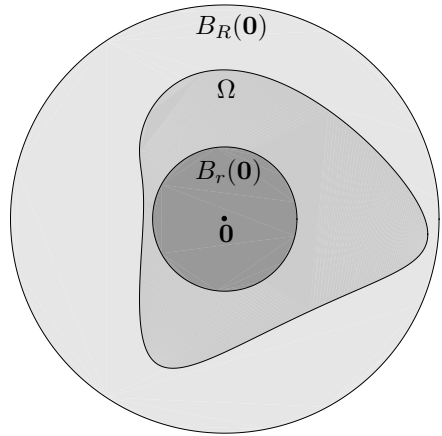}}
\put(2.92,2.69){\footnotesize $\b0$}
\put(2.6,3.53){\footnotesize $B_r(\b0)$}
\put(2.89,4.61){\footnotesize $\Omega$}
\put(2.6,5.46){\footnotesize $B_R(\b0)$}
\end{picture}
\caption{
A sketch of sets introduced in the proof of Theorem \ref{th:stability}.
\label{fig:schemStab}
} 
\end{center}
\end{figure}

\begin{proof}[Proof of Theorem \ref{th:stability}]
Since $\b0 \in {\rm int}(\Omega)$ there exists $r > 0$ such that $B_r(\b0) \subset {\rm int}(\Omega)$.
Since $\Omega$ is bounded there exists $R > r$ such that $\Omega \subset B_R(\b0)$,
see Fig.~\ref{fig:schemStab}.

We first show \eqref{it:1} $\Rightarrow$ \eqref{it:2}.
If $\b0$ is asymptotically stable
there exists $\delta > 0$ such that
$g^i(x) \to \b0$ as $i \to \infty$ for all $x \in B_\delta(\b0)$.
Since $g$ is linearly homogeneous, this is also true for all $x \in B_R(\b0)$.
By Lemma \ref{le:stability}\eqref{it:uniformConvergence}
there exists $m \ge 1$ such that $g^m \left( B_R(\b0) \right) \subset B_r(\b0)$.
Then \eqref{it:2} is satisfied because $B_r(\b0) \subset {\rm int}(\Omega)$.

Observe \eqref{it:2} $\Rightarrow$ \eqref{it:3} is trivial because
${\rm int}(\Omega) \subset {\rm int} \left( \bigcup_{i=0}^{m-1} g^i(\Omega) \right)$.
Thus it remains to show \eqref{it:3} $\Rightarrow$ \eqref{it:1}.
Let $\Xi = \bigcup_{i=0}^{m-1} g^i(\Omega)$.
Statement \eqref{it:3} implies there exists $0 < \alpha < 1$ such that $g^m(\Omega) \subset \alpha \Xi$
(where $\alpha \Xi = \left\{ \alpha x \,\middle|\, x \in \Xi \right\}$)\removableFootnote{
In \cite{Si16d} using the approach of \cite{Ga92},
the existence of $\alpha$ is obvious because we
just deal with the line segment from $(0,1)$ to $(1,0)$.
In \cite{AtLa14} such an $\alpha$ is provided by assumption.
}.
This is because if there does not exist such an $\alpha$,
then for each $\alpha_j = 1 - \frac{1}{j}$ there exists $x_j \in g^m(\Omega)$ with $x_j \notin \alpha_j \Xi$.
The sequence $\{ x_j \}$ is bounded so has a subsequence converging to some $x^*$
(by the Bolzano-Weierstrass Theorem)
and $x^* \in g^m(\Omega)$ because $g^m(\Omega)$ is closed and $x^* \notin {\rm int}(\Xi)$
which contradicts \eqref{it:3}.

We now show $g^m(\Xi) \subset \alpha \Xi$.
Choose any $x \in \Xi$.
Then $x \in g^i(\Omega)$ for some $i \in \{ 0,\ldots,m-1 \}$ and $g^{m-i}(x) \in g^m(\Omega)$.
Let $y = \frac{x}{\alpha}$.
Since $g$ is linearly homogeneous,
$g^{m-i}(y) = \frac{1}{\alpha} g^{m-i}(x) \in \frac{1}{\alpha} g^m(\Omega) \subset \Xi$.
But $\Xi$ is forward invariant by \eqref{it:3}, thus $g^m(y) \in \Xi$.
Then $g^m(x) = \alpha g^m(y) \in \alpha \Xi$ as required.

Then $g^{k m}(\Xi) \subset \alpha^k \Xi$ for all $k \ge 1$
(again using the linear homogeneity of $g$).
Thus $g^n(x) \to \b0$ as $n \to \infty$ for all $x \in \Xi$.
In particular, $g^n(x) \to \b0$ for all $x \in B_r(\b0)$,
thus $\b0$ is asymptotically stable by Lemma \ref{le:stability}\eqref{it:asyStable}.
\end{proof}

\subsection{Numerical implementation}
\label{sub:OmegaNumerically}

Here we outline a numerical procedure, based on Theorem \ref{th:stability},
that can verify the stability of a fixed point at a BCB.
For a given BCB this could provide a computed-assisted proof of stability.

We first form the piecewise-linear approximation \eqref{eq:g}.
Under such a map the image of a {\em polytope}
(a region bounded by flat surfaces, i.e.~a generalisation of polygons to more than two dimensions) is another polytope.
If a polytope has finitely many vertices (so can be encoded with finitely many points)
its image will also have finitely many vertices.
Hence we choose $\Omega$ to be a reasonably simple polytope.
We then iteratively compute $g^m(\Omega)$ and $\bigcup_{i=0}^{m-1} g^i(\Omega)$
(eliminating redundant vertices at each step)
and if we find $m$ for which condition \eqref{it:3} of Theorem \ref{th:stability} is satisfied
we conclude that $\b0$ is an asymptotically stable fixed point of \eqref{eq:g}.
Moreover, $\b0$ will be an asymptotically stable fixed point of
\eqref{eq:fPWS} at the BCB (by Lemma \ref{le:approx}).

If we cannot find $m$ satisfying condition \eqref{it:3} by iterating up to
some maximum permitted value $m_{\rm max}$,
then we cannot make a conclusion about the stability of $\b0$.
However if the value of $m_{\rm max}$ is reasonably large
this may provide useful evidence that $\b0$ is in fact unstable.

The encoding of $\Omega$ and its images can further be aided by the following observation.
A set $\Omega \subset \mathbb{R}^d$ is
{\em radially convex} (or {\em star-shaped} with respect to the origin)
if for every $x \in \Omega$ and $0 < \alpha < 1$ we have $\alpha x \in \Omega$.
It is a simple exercise to show that
the property of radial convexity is preserved under linearly homogeneous maps\removableFootnote{
Proof:
Choose any $x \in f(\Omega)$ and any $0 < \alpha < 1$.
Let $y \in \Omega$ be such that $x = f(y)$.
Then $\alpha y \in \Omega$ because $\Omega$ is radially convex.
Also $f(\alpha y) = \alpha f(y) = \alpha x$, because $f$ is linearly homogeneous,
thus $\alpha x \in f(\Omega)$ as required.
}.

\begin{figure}[b!]
\begin{center}
\setlength{\unitlength}{1cm}
\begin{picture}(12,6)
\put(0,0){\includegraphics[height=6cm]{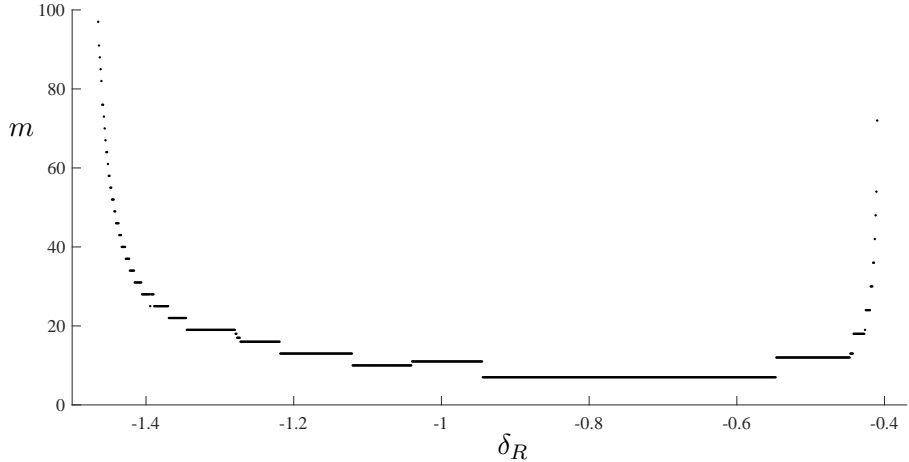}}
\put(6.5,0){\small $\delta_R$}
\put(0,4.23){\small $m$}
\end{picture}
\caption{
The smallest value of $m$ for which condition \eqref{it:3} of Theorem \ref{th:stability}
is satisfied for \eqref{eq:g} with \eqref{eq:ALAR} and \eqref{eq:params},
where $\Omega$ is the square with vertices $(\pm 1,0)$ and $(0,\pm 1)$.
\label{fig:stabIterDiamondMany}
} 
\end{center}
\end{figure}

\begin{figure}[b!]
\begin{center}
\setlength{\unitlength}{1cm}
\begin{picture}(12.5,6.45)
\put(0,0){\includegraphics[height=6cm]{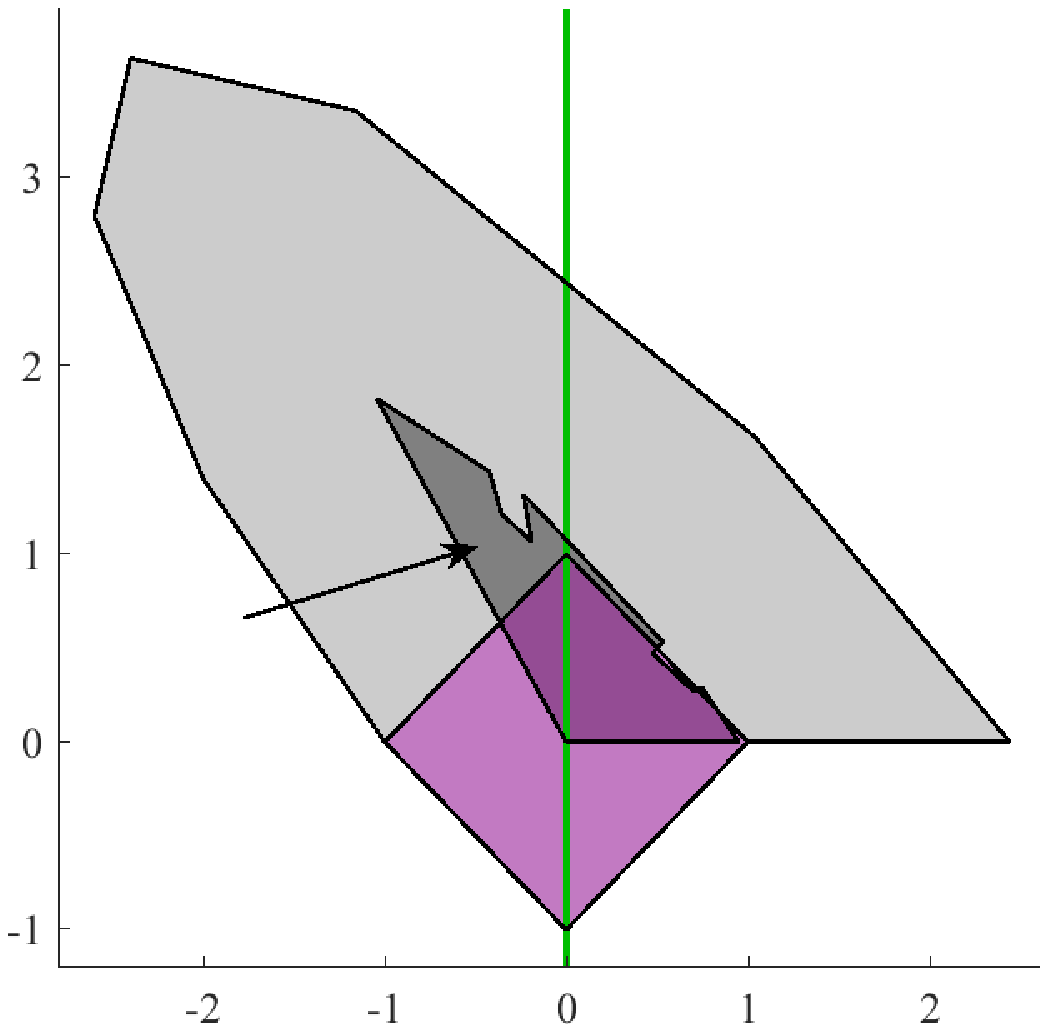}}
\put(6.5,0){\includegraphics[height=6cm]{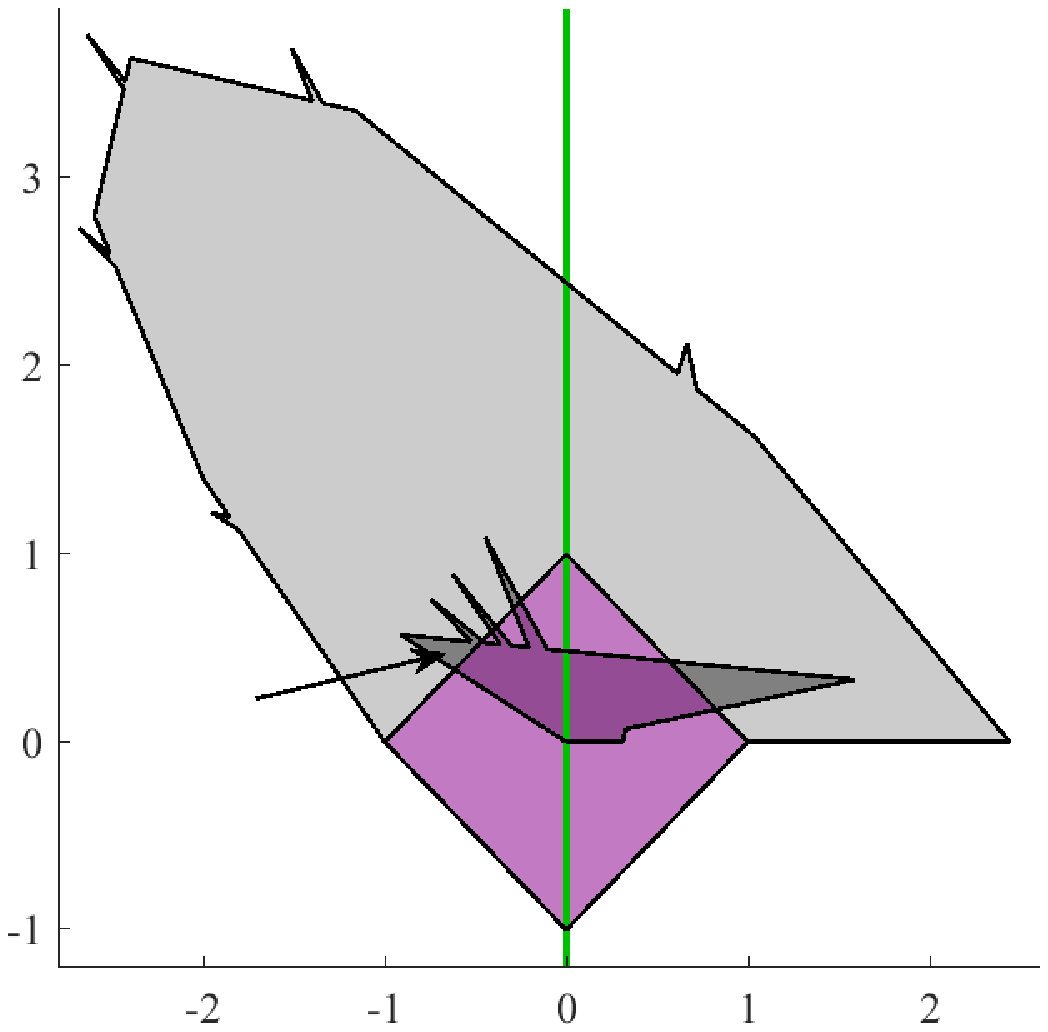}}
\put(2.15,6.15){\small ${\rm {\bf a})}~~\delta_R = -1.4$}
\put(3.2,0){\small $x_1$}
\put(0,3.42){\small $x_2$}
\put(1.38,4.52){\scriptsize $\bigcup\limits_{\scriptscriptstyle i=0}^{\scriptscriptstyle m-1} g^i(\Omega)$}
\put(.76,2.48){\scriptsize $g^m(\Omega)$}
\put(3.45,5.2){\footnotesize $\Sigma$}
\put(2.86,1.55){\footnotesize $\Omega$}
\put(8.65,6.15){\small ${\rm {\bf b})}~~\delta_R = -0.5$}
\put(9.7,0){\small $x_1$}
\put(6.5,3.42){\small $x_2$}
\put(7.88,4.52){\scriptsize $\bigcup\limits_{\scriptscriptstyle i=0}^{\scriptscriptstyle m-1} g^i(\Omega)$}
\put(7.32,2.04){\scriptsize $g^m(\Omega)$}
\put(9.95,5.2){\footnotesize $\Sigma$}
\put(9.36,1.55){\footnotesize $\Omega$}
\end{picture}
\caption{
Phase portraits illustrating condition \eqref{it:3} of Theorem \ref{th:stability}
for the map of Fig.~\ref{fig:stabIterDiamondMany} for two different values of $\delta_R$.
In the left plot $m = 28$; in the right plot $m = 12$.
\label{fig:stabIterDiamondQQ}
} 
\end{center}
\end{figure}

As an example we consider \eqref{eq:g} in two-dimensions with
$A_L$ and $A_R$ given by \eqref{eq:ALAR}.
We fix
\begin{align}
\tau_L &= 2, &
\delta_L &= 1.4, &
\tau_R &= -0.8,
\label{eq:params}
\end{align}
and consider several different values of $\delta_R$.
We let $\Omega$ be the square with vertices $(\pm 1,0)$ and $(0,\pm 1)$.
Fig.~\ref{fig:stabIterDiamondMany} shows the smallest value of $m$
for which condition \eqref{it:3} is satisfied
for $1000$ different values of $\delta_R$.
Formally we can only conclude that $\b0$ is asymptotically stable
for the discrete set of $1000$ parameter values that were simulated,
but the results strongly suggest that $\b0$ is asymptotically stable for all $-1.46 \le \delta_R \le -0.41$.

With $\delta_R = -1.4$, for example, condition \eqref{it:3} is first satisfied with $m = 28$
and the sets $g^m(\Omega)$ and $\bigcup_{i=0}^{m-1} g^i(\Omega)$
are shown in Fig.~\ref{fig:stabIterDiamondQQ}a for this value of $m$.
With instead $\delta_R = -0.5$ we obtain $m = 12$, see Fig.~\ref{fig:stabIterDiamondQQ}b.
Certainly with alternate $\Omega$
condition \eqref{it:3} can be satisfied at smaller values of $m$
but we do not attempt this here.
An understanding of the bifurcations at $\delta_R \approx -1.46$ and $\delta_R \approx -0.41$
at which stability is lost is beyond the scope of this paper.

\subsection{Persistence of an local attractor}
\label{sub:persistence}

Here we show that if the fixed point at a BCB is asymptotically stable
then there exists a local attractor for nearby values of the parameters.

\begin{theorem}
Let $f_\mu$ be a continuous map on $\mathbb{R}^d$ with parameter $\mu \in \mathbb{R}$.
Suppose $\b0$ is an asymptotically stable fixed point of $f_0$ and let
$\Omega \subset \mathbb{R}^d$ be a compact set in its basin of attraction.
Then for all $\ee > 0$ there exists $\delta > 0$ and $N \in \mathbb{Z}$ such that
\begin{equation}
f_\mu^n(x) \in B_\ee(\b0), \quad {\rm for~all~} x \in \Omega,\, n \ge N, {\rm ~and~}
\mu \in (-\delta,\delta).
\label{eq:attractingSet}
\end{equation}
\label{th:attractingSet}
\end{theorem}

Theorem \ref{th:attractingSet} is the discrete-time analogue of Theorem 5.4 of \cite{DiNo08} for
ordinary differential equations
and is proved below using standard techniques in real analysis.
Specifically it shows that all points in a compact set
$\Omega \subset \mathbb{R}^d$ map into, and never escape from,
the ball $B_\ee(\b0)$ within at most $N$ iterations of $f$.
Therefore $\Omega$ is a trapping region for $f^N$, assuming $B_\ee(\b0) \subset \Omega$,
and thus $f$ has an attractor in $B_\ee(\b0)$.

\begin{proof}[Proof of Theorem \ref{th:attractingSet}]
Choose any $\ee > 0$.
Since $\b0$ is an asymptotically stable fixed point of $f_0$
there exists $\delta_1 > 0$ (with $\delta_1 \le \ee$) such that
\begin{equation}
f_0^n(x) \in B_{\frac{\ee}{2}}(\b0), \quad
{\rm for~all~} x \in B_{\delta_1}(\b0), {\rm ~and~all~} n \ge 0,
\label{eq:LyapunovStable}
\end{equation}
and $f_0^n(x) \to \b0$ as $n \to \infty$, for all $x \in {\rm cl} \left( B_{\delta_1}(\b0) \right)$.
This convergence is uniform because ${\rm cl} \left( B_{\delta_1}(\b0) \right)$ is compact
(this is a consequence of the Arzel\`{a}-Ascoli theorem;
for details refer to the proof of Lemma 3 of \cite{Si16d}).
Thus there exists $M \in \mathbb{Z}$ such that
\begin{equation}
f_0^M(x) \in B_{\frac{\delta_1}{2}}(\b0), \quad {\rm for~all~} x \in B_{\delta_1}(\b0).
\label{eq:uniformConvergenceB}
\end{equation}
Also $f_0^n(x) \to \b0$ as $n \to \infty$, for all $x \in \Omega$,
because $\Omega$ belongs to the basin of attraction of $\b0$.
This convergence is similarly uniform because $\Omega$ is compact.
Thus there exists $N \in \mathbb{Z}$ such that
\begin{equation}
f_0^N(x) \in B_{\frac{\delta_1}{2}}(\b0), \quad {\rm for~all~} x \in \Omega.
\label{eq:uniformConvergenceK}
\end{equation}
Since $f_\mu^N$ is continuous and $\Omega$ is compact, $f_\mu^N$ is uniformly continuous on $\Omega$.
Thus there exists $\delta_2 > 0$ such that
\begin{equation}
\left\| f_\mu^N(x) - f_0^N(x) \right\| < \frac{\delta_1}{2} \;, \quad
{\rm for~all~} x \in \Omega, {\rm ~and~all~} \mu \in (-\delta_2,\delta_2).
\label{eq:uniformContinuityK}
\end{equation}
Similarly $f_\mu^i$ is uniformly continuous on $B_{\delta_1}(\b0)$ for all $i \ge 1$,
thus there exists $\delta > 0$ (with $\delta \le \delta_2$) such that
\begin{equation}
\left\| f_\mu^i(x) - f_0^i(x) \right\| < \frac{\delta_1}{2} \;, \quad
{\rm for~all~} x \in B_{\delta_1}(\b0), {\rm ~all~} i \in \{ 1,\ldots,M \}, {\rm ~and~all~}
\mu \in (-\delta,\delta).
\label{eq:uniformContinuityB}
\end{equation}
Now choose any $x \in \Omega$ and any $\mu \in (-\delta,\delta)$.
By \eqref{eq:uniformConvergenceK} and \eqref{eq:uniformContinuityK}, we have
\begin{equation}
\left\| f_\mu^N(x) \right\| \le
\left\| f_\mu^N(x) - f_0^N(x) \right\| + \left\| f_0^N(x) \right\| \le
\frac{\delta_1}{2} + \frac{\delta_1}{2} = \delta_1 \,.
\nonumber
\end{equation}
This verifies \eqref{eq:attractingSet} for $n = N$ (because $\delta_1 \le \ee$).
By applying \eqref{eq:LyapunovStable} and \eqref{eq:uniformContinuityB} to the point $f_\mu^N(x)$,
for any $i \in \{ 1,\ldots,M \}$ we have
\begin{equation}
\left\| f_\mu^{N+i}(x) \right\| \le
\left\| f_\mu^{N+i}(x) - f_0^{N+i}(x) \right\| + \left\| f_0^{N+i}(x) \right\| \le
\frac{\delta_1}{2} + \frac{\ee}{2} \le \ee.
\nonumber
\end{equation}
This verifies \eqref{eq:attractingSet} for $n = N+1,\ldots,N+M$.
By applying \eqref{eq:uniformConvergenceB} and \eqref{eq:uniformContinuityB} to $f_\mu^N(x)$, we have
\begin{equation}
\left\| f_\mu^{N+M}(x) \right\| \le
\left\| f_\mu^{N+M}(x) - f_0^{N+M}(x) \right\| + \left\| f_0^{N+M}(x) \right\| \le
\frac{\delta_1}{2} + \frac{\delta_1}{2} = \delta_1 \,.
\nonumber
\end{equation}
Thus we can repeat the last two steps indefinitely to
verify \eqref{eq:attractingSet} for all $n \ge N$.
\end{proof}

\section{Lyapunov exponents}
\label{sec:lyap}
\setcounter{equation}{0}

\subsection{Lyapunov exponents for smooth maps}
\label{sub:lyapSmooth}

Let $f$ be a $C^1$ map on $\mathbb{R}^d$.
Given $x \in \mathbb{R}^d$, $v \in \rT \mathbb{R}^d$, and $0 < \delta \ll 1$,
we are interested in how the forward orbit of the perturbed point $x + \delta v$
compares to the forward orbit of $x$.
After one application of $f$ the difference between the iterates is
\begin{equation}
f(x + \delta v) - f(x) = \delta \rD f(x) v + \co(\delta).
\label{eq:Fperturbation}
\end{equation}
After $n$ applications of $f$ the difference is
\begin{equation}
f^n(x + \delta v) - f^n(x) = \delta \rD f^n(x) v + \co(\delta).
\label{eq:Fnperturbation}
\end{equation}
For large $n$ the distance between the iterates is\removableFootnote{
This is consequence of the trivial rearrangement
$\left\| f^n(x + \delta v) - f^n(x) \right\| = \delta
\,{\rm e}^{n \frac{1}{n} \ln \left(  \left\| \rD f^n(x) v \right\| + \co(1) \right)}$.
}
$\left\| f^n(x + \delta v) - f^n(x) \right\| \sim \delta {\rm e}^{\lambda(x,v) n}$, where
\begin{equation}
\lambda(x,v) = \lim_{n \to \infty} \frac{1}{n} \ln \left( \left\| \rD f^n(x) v \right\| \right),
\label{eq:lyapSmooth}
\end{equation}
if this limit exists.
The {\em Lyapunov exponent} $\lambda(x,v)$ represents the asymptotic rate of expansion of $f$ at $x \in \mathbb{R}^d$
in the direction $v \in \rT \mathbb{R}^d$.
Oseledets' theorem \cite{BaPe07,EcRu85,Vi14} gives conditions under which,
for almost all $x$ in an invariant set,
$\lambda(x,v)$ is well-defined and takes at most $d$ values independent of $x$\removableFootnote{
Specifically, the limit $Z = \lim_{n \to \infty} \left( \rD f^n(x)^{\sf T} \rD f^n(x) \right)^{\frac{1}{2 n}}$ exists.
Note: for any matrix $A$, the matrix $A^{\sf T} A$ is positive semi-definite
(as stated on the ``positive-definite matrix'' Wikipedia page),
and thus has a unique power of $\frac{1}{2 n}$ that is also positive semi-definite
(obtained by diagonalisation, what if it is not diagonalisable?).
The natural logarithms of the eigenvalues of $Z$ are the Lyapunov exponents.
The eigenspaces provide the corresponding direction vectors.
These give the directions of stable and unstable manifolds, see \cite{EcRu85}.
}\removableFootnote{
Cocycles have a `multiplicative' property,
thus Oseledets' theorem is known as a `multiplicative ergodic theorem'.
In contrast, Kingman's subadditive ergodic theorem
applies to functions that have a `subadditive' property.
Also the Furstenberg-Kesten theorem applies to
Lyapunov exponents that are built upon random matrix products \cite{Ar98,FuKe60}.
}.

Before we consider piecewise-smooth maps,
it is helpful to consider the evolution of points and tangent vectors together.
Define a map $h$ on the tangent bundle $\mathbb{R}^d \times \rT \mathbb{R}^d$ by
\begin{equation}
h(x,v) = \big( f(x), \rD f(x) v \big).
\label{eq:hSmooth}
\end{equation}
Since $\rD f^n(x) = \rD f \big( f^{n-1}(x) \big) \rD f \big( f^{n-2}(x) \big) \cdots \rD f(x)$,
the composition of $h$ with itself $n$ times is
\begin{equation}
h^n(x,v) = \big( f^n(x), \rD f^n(x) v \big).
\label{eq:hSmoothn}
\end{equation}
The first component of \eqref{eq:hSmoothn} is the $n^{\rm th}$ iterate of $x$ under $f$.
The second component provides the vector in \eqref{eq:lyapSmooth}.
The matrix $\rD f^n(x)$ is a {\em cocycle} because
$\rD f^{m+n}(x) = \rD^m(f^n(x)) \rD f^n(x)$ for all $m,n \ge 0$, \cite{KaHa95}\removableFootnote{
Also $h$ is an example of a {\em skew product},
although there doesn't seem to be much point mentioning this
and I don't have a good reference for it anyway.
}.

\subsection{Lyapunov exponents for piecewise-smooth maps}
\label{sub:lyapPWS}

Now let $f$ be a continuous, piecewise-smooth map of the form \eqref{eq:fPWS}
(here we ignore the dependency on $\mu$).
We assume $f_L$ and $f_R$ are $C^1$ throughout 
$\left\{ x \in \mathbb{R}^d \,\middle|\, x_1 \le 0 \right\}$
and $\left\{ x \in \mathbb{R}^d \,\middle|\, x_1 \ge 0 \right\}$, respectively,
and first generalise \eqref{eq:Fperturbation} to this piecewise-smooth setting.

\begin{lemma}
For any $x \in \mathbb{R}^d$ and $v \in \rT \mathbb{R}^d$,
\begin{equation}
f(x + \delta v) - f(x) = \delta C(x,v) v + \co(\delta),
\label{eq:perturbation1}
\end{equation}
where
\begin{equation}
C(x,v) = \begin{cases}
\rD f_L(x), & x_1 < 0, \!{\rm ~or~} x_1 = 0 {\rm ~and~} v_1 < 0, \\
\rD f_R(x), & x_1 > 0, \!{\rm ~or~} x_1 = 0 {\rm ~and~} v_1 \ge 0,
\end{cases}
\label{eq:C}
\end{equation}
\label{le:C}
\end{lemma}

\begin{proof}
If $x_1 > 0$ then $(x + \delta v)_1 > 0$ for sufficiently small $\delta > 0$.
Thus $f(x) = f_R(x)$, $f(x + \delta v) = f_R(x + \delta v)$,
and the differentiability of $f_R$ gives \eqref{eq:perturbation1} with $C(x,v) = \rD f_R(x)$.
If $x_1 < 0$ the same arguments apply to $f_L$.

If $x_1 = 0$ and $v_1 \ge 0$ then $(x + \delta v)_1 \ge 0$.
Since $f$ is continuous on $\Sigma$, we again have
$f(x) = f_R(x)$, $f(x + \delta v) = f_R(x + \delta v)$,
and thus \eqref{eq:perturbation1} with $C(x,v) = \rD f_R(x)$.
If $x_1 = 0$ and $v_1 < 0$ the same arguments apply to $f_L$.
\end{proof}

In \eqref{eq:C} our choice of $\rD f_R(x)$ in the special case $x_1 = v_1 = 0$ is immaterial
because although in general $\rD f_L(x)$ and $\rD f_R(x)$ are different matrices,
the continuity of $f$ implies $\rD f_L(x) v = \rD f_R(x) v$.
As an example, Fig.~\ref{fig:unitCircleImage} shows
the set $\left\{ C(\b0,v) v \,\middle|\, v \in \mathbb{S}^1 \right\}$
for the border-collision normal form with the parameter values of Fig.~\ref{fig:stabIterDiamondQQ}a.

\begin{figure}[b!]
\begin{center}
\setlength{\unitlength}{1cm}
\begin{picture}(8,6)
\put(0,0){\includegraphics[width=8cm]{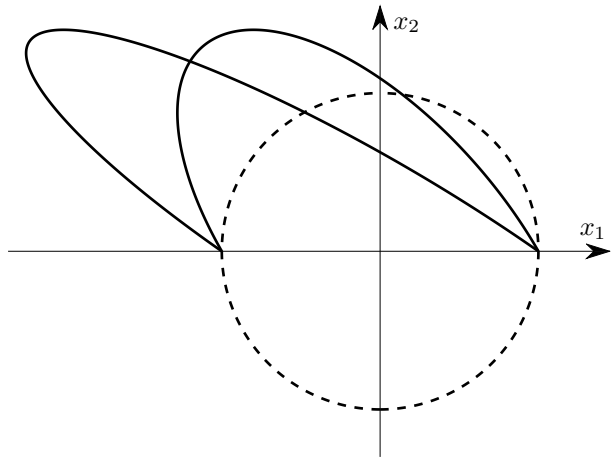}}
\put(7.6,2.95){\footnotesize $x_1$}
\put(5.12,5.7){\footnotesize $x_2$}
\end{picture}
\caption{
The unit circle (dashed) and its image (solid) under the tangent map $v \mapsto C(x,v) v$ with $x = \b0$
for the map \eqref{eq:g} with \eqref{eq:ALAR}, \eqref{eq:params}, and $\delta_R = -1.4$.
\label{fig:unitCircleImage}
} 
\end{center}
\end{figure}

Analogous to \eqref{eq:hSmooth}, we define
\begin{equation}
h(x,v) = \big( f(x), C(x,v) v \big).
\label{eq:h}
\end{equation}
Evidently we have
\begin{equation}
h^n(x,v) = \big( f^n(x), C_n(x,v) v \big),
\label{eq:hn}
\end{equation}
where
\begin{equation}
C_n(x,v) = C \left( h^{n-1}(x,v) \right) C \left( h^{n-2}(x,v) \right) \cdots C(x,v).
\label{eq:Cn}
\end{equation}
This is a cocycle because $C_{m+n}(x,v) = C_m(h^n(x,v)) C_n(x,v)$
for all $m,n \ge 0$.
We now show that $C_n(x,v)$ plays the desired role of $\rD f^n(x)$\removableFootnote{
If $f$ is piecewise-linear with fixed point $x = \b0$, then $C_n(\b0,v) v = f^n(v)$.
I won't mention this because it's super confusing.
}.

\begin{lemma}
For any $x \in \mathbb{R}^d$ and $v \in \rT \mathbb{R}^d$,
\begin{equation}
f^n(x + \delta v) - f^n(x) = \delta C_n(x,v) v + \co(\delta),
\label{eq:perturbationn}
\end{equation}
for all $n \ge 1$.
\label{le:Cn}
\end{lemma}

\begin{proof}
The result is true for $n = 1$ by Lemma \ref{le:C}.
Suppose the result is true for some $n = k \ge 1$.
Then
\begin{equation}
f^k(x + \delta v) - f^k(x) = \delta C_k(x,v) v + \co(\delta),
\nonumber
\end{equation}
which implies
\begin{equation}
f^{k+1}(x + \delta v) - f^{k+1}(x) =
f \left( f^k(x) + \delta C_k(x,v) v \right) - f^{k+1}(x) + \co(\delta),
\nonumber
\end{equation}
using also the continuity of $\rD f$.
By then applying Lemma \ref{le:C} to $h^k(x,v)$ we obtain
\begin{equation}
f^{k+1}(x + \delta v) - f^{k+1}(x) =
\delta C(h^k(x,v)) C_k(x,v) v + \co(\delta),
\nonumber
\end{equation}
and thus, by the cocycle property, the result is also true for $n = k+1$.
Hence the result is true for all $n \ge 1$ by induction.
\end{proof}

In view of Lemma \ref{le:Cn}, we define
\begin{equation}
\lambda(x,v) = \lim_{n \to \infty} \frac{1}{n} \ln \left( \left\| C_n(x,v) v \right\| \right),
\label{eq:lyapPWS}
\end{equation}
if the limit exists.
Oseledets' theorem does not apply to \eqref{eq:lyapPWS}
because the cocycle $C_n(x,v)$ is dependent on $v$.
Indeed it is not difficult to find examples for which $\lambda(x,v)$ takes more than $d$ distinct values.
Consider for instance the one-dimensional map
\begin{equation}
g(x) = \begin{cases}
a_L x, & x \le 0, \\
a_R x, & x \ge 0,
\end{cases}
\label{eq:gEx1}
\end{equation}
where $a_L, a_R > 0$, see Fig.~\ref{fig:cobweb1dEx}.
For the fixed point $x=0$, we have $\lambda(0,-1) = \ln(a_L)$ and $\lambda(0,1) = \ln(a_R)$.
Thus if $a_L \ne a_R$ then \eqref{eq:lyapPWS} takes two different values at $x=0$.

\begin{figure}[b!]
\begin{center}
\setlength{\unitlength}{1cm}
\begin{picture}(6,6)
\put(0,0){\includegraphics[width=6cm]{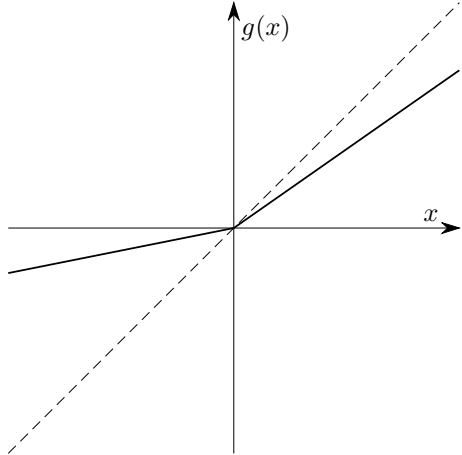}}
\put(5.52,3.09){\footnotesize $x$}
\put(3.1,5.56){\footnotesize $g(x)$}
\end{picture}
\caption{
A sketch of the one-dimensional piecewise-linear map \eqref{eq:gEx1}.
\label{fig:cobweb1dEx}
} 
\end{center}
\end{figure}

For the two-dimensional map \eqref{eq:g} with \eqref{eq:ALAR}
it was found in \cite{Si16d} that
the normalised tangent map $v \mapsto \frac{C(x,v) v}{\| C(x,v) v \|}$ on $\mathbb{S}^1$
appears to be chaotic with infinitely many ergodic invariant probability measures
each generating a potentially distinct value for \eqref{eq:lyapPWS} at $x = \b0$.

\subsection{Determinant-based bounds on the maximal Lyapunov exponent}
\label{sub:bound}

Let $f_\mu$ be a continuous, piecewise-$C^1$ map of the form \eqref{eq:fPWS}.
We have shown that if $x = \b0$ is an asymptotically stable fixed point of $f_0$
then $f_\mu$ has an attractor $\Lambda$ near $x = \b0$ when $|\mu|$ is small (see Theorem \ref{th:attractingSet}).
Here we construct a lower bound for the maximum value of the Lyapunov exponent \eqref{eq:lyapPWS}
for $x \in \Lambda$.

Let $\Omega \subset \mathbb{R}^d$ be a compact set and let
\begin{equation}
\begin{split}
a_L &= \min \left\{ \left| \det \left( \rD f_{L,\mu}(x) \right) \right| \,\big|\,
x \in \Omega,\, x_1 \le 0 \right\}, \\
a_R &= \min \left\{ \left| \det \left( \rD f_{R,\mu}(x) \right) \right| \,\big|\,
x \in \Omega,\, x_1 \ge 0 \right\}.
\end{split}
\label{eq:aLaR}
\end{equation}
The idea is that $\Omega$ is small, so $a_L$ and $a_R$ are usefully approximated
by $|\det(A_L)|$ and $|\det(A_R)|$, and $\Lambda \subset \Omega$,
so results for $x \in \Omega$ immediately apply to $x \in \Lambda$.
In fact for the purposes of showing that a chaotic attractor is created in the BCB at $\mu$ = 0,
by Theorem \ref{th:attractingSet} we can make $\Omega$ as small as we like
(and still allow $\mu \ne 0$) hence $a_L$ and $a_R$ can be made to be arbitrarily close to
$|\det(A_L)|$ and $|\det(A_R)|$.

Given $x \in \Omega$, let $\ell_n$ [resp.~$r_n$]
be the number of iterates $i \in \{ 0,\ldots,n-1 \}$ for which $f^i(x)_1 < 0$ [resp.~$f^i(x)_1 > 0$].
Also let
\begin{align}
\ell &= \liminf_{n \to \infty} \frac{\ell_n}{n}, &
r &= \liminf_{n \to \infty} \frac{r_n}{n}.
\label{eq:lr}
\end{align}
These quantities are $x$-dependent but in our notation we have
omitted this dependency for brevity.

\begin{theorem}
Suppose $f^i(x) \in \Omega$ for all $i \ge 0$.
Let
\begin{equation}
a = \min \left[ a_L^\ell a_R^{1-\ell},
a_L^{1-r} a_R^r \right],
\label{eq:detbound}
\end{equation}
and suppose $a > 0$.
Let
\begin{equation}
\lambda_{\rm bound} = \frac{1}{d} \ln \left( 2^{\ell+r-1} a \right).
\label{eq:lambdabound}
\end{equation}
Then
\begin{equation}
\liminf_{n \to \infty} \frac{1}{n} \ln \left( \left\| C_n(x,v) v \right\| \right) \ge \lambda_{\rm bound} \,,
\label{eq:mainBound}
\end{equation}
for almost all $v \in T \mathbb{R}^d$.
\label{th:main}
\end{theorem}

A proof of Theorem \ref{th:main} follows some technical remarks.
If $\lambda_{\rm bound} > 0$ and the limit \eqref{eq:lyapPWS} exists,
then $\lambda(x,v) > 0$ for some (in fact almost all) $v$,
and so the maximal Lyapunov exponent is positive.
If $\lambda_{\rm bound} > 0$ but the limit \eqref{eq:lyapPWS} does not exist,
\eqref{eq:mainBound} still ensures the dynamics is locally expanding by Lemma \ref{le:Cn}.

\begin{remark}
If $\nu$ is the invariant probability measure associated with an attractor $\Lambda \subset \Omega$,
we often have $\nu(\Sigma) = 0$ (exceptions include periodic solutions with one or more points on $\Sigma$).
In this case $\frac{\ell_n + r_n}{n} \to 1$ as $n \to \infty$ for typical $x \in \Lambda$.
Then usually $\ell + r = 1$\removableFootnote{
If $\nu$ is ergodic I think we must have $\ell + r = 1$ for $\nu$ almost all $x$
but I haven't quite been able to prove it
(note that if $\nu = \delta_\b0$ and $f^i(x) \to \b0$ asymptotically
we may have $\ell + r = 1$ even though $\nu(\Sigma) \ne 0$,
but in this example $x$ is not $\nu$-typical).

Here is my attempted proof of the contrapositive:
For each $k \ge 1$ we define a continuous function $\varphi_k : \mathbb{R}^d \to \mathbb{R}$ by
\begin{equation}
\varphi_k(x) = \begin{cases}
1 + k x_1 \,, & -\frac{1}{k} \le x_1 \le 0, \\
1 - k x_1 \,, & 0 \le x_1 \le \frac{1}{k}, \\
0, & {\rm otherwise}.
\end{cases}
\nonumber
\end{equation}
By Birkhoff's ergodic theorem \cite{Wa82}
\begin{equation}
\int_{\mathbb{R}^d} \varphi_k \,d\nu = \lim_{n \to \infty} \frac{1}{n}
\sum_{i=0}^{n-1} \varphi_k \left( f^i(x) \right).
\label{eq:Birkhoff}
\end{equation}
If $\ell + r < 1$, then there are infinitely many values of $n$
for which $\frac{\ell_n + r_n}{n}$ is less than some value less than $1$, call it $1 - \tilde{q}$.
For each of these values of $n$ we have
$\frac{1}{n} \sum_{i=0}^{n-1} \varphi_k \left( g^i(x) \right) \ge \tilde{q}$ for all $k$.
Thus the right hand-side of \eqref{eq:Birkhoff} is nonzero for all $k$.
But the left hand-side of \eqref{eq:Birkhoff} limits to $\nu(\Sigma)$ as $k \to \infty$
(except I don't know how to show this!).
}
which implies $a = a_L^\ell a_R^r$.
Here $a$ is the weighted geometric mean of $a_L$ and $a_R$ induced by the forward orbit of $x$.
Then $\lambda_{\rm bound} = \frac{1}{d} \ln(a)$,
therefore $a > 1$ implies $\lambda(x,v) > 0$ (assuming \eqref{eq:lyapPWS} exists).
\label{re:bound2}
\end{remark}

\begin{remark}
In practice, for a given map one may be able to obtain estimates on $\ell$ and $r$ throughout a set $\Omega$ \cite{Gl15b}.
For example if it is not possible for three consecutive iterates to lie in $x_1 < 0$
(i.e.~$LLL$ is a {\em forbidden word} \cite{LiMa95}) then $\ell \le \frac{2}{3}$.
If it is only known that $\ell + r = 1$, then $a = a_L^\ell a_R^r \ge \min(a_L,a_R)$.
So if $a_L$ and $a_R$ are both greater than $1$ we must have $\lambda(x,v) > 0$ (assuming \eqref{eq:lyapPWS} exists)
which makes sense because in this case the map is area-expanding
on both sides of $\Sigma$.
\label{re:bound3}
\end{remark}

\begin{remark}
It is instructive to apply Theorem \ref{th:main} to $x = \b0$ when this point is a fixed point.
In this case $\ell = r = 0$, so $a = \min(a_L,a_R)$ and
$\lambda_{\rm bound} = \frac{1}{d} \ln \left( \frac{a}{2} \right)$.
Thus even if $a_L$ and $a_R$ are both greater than $1$ (but not both greater than $2$)
it is possible to have $\lambda(x,v) < 0$
and thus possible for the fixed point to be stable (as in Fig.~\ref{fig:stabIterDiamondQQ}a).
\end{remark}

Our proof of Theorem \ref{th:main} is based on the following
upper bound for the measure of the set of unit tangent vectors $v$
for which $\| A v \|$ is `small', where $A$ is a matrix.
The bound is crude but far-reaching and well-suited for Theorem \ref{th:main}
because the only information of $A$ that is used is its determinant.

\begin{lemma}
Let $A$ be a real-valued $d \times d$ matrix with $\det(A) \ne 0$, and let $c > 0$.
Then\removableFootnote{
If $d = 2$ and $|\det(A)| > 1$, numerical results
(see {\sc goMeasureContractingExample2.m}) suggest the stronger result
\begin{equation}
\sphMeas \left( \left\{ v \in \mathbb{S}^{d-1} \,\middle|\,
\| A v \| \le c \right\} \right) \le \frac{c^d}{2 |\det(A)|}.
\nonumber
\end{equation}
However, this wouldn't lead to an improvement on Theorem \ref{th:main}
because it's still a linear dependency on $\frac{1}{|\det(A)|}$.
It would be interesting to see if the factor of $2$ changes to something else for $d > 2$.
It would also be interesting to search for the optimal bound as a function of $|\det(A)|$;
I think this is doable with $d = 2$ by directly going through the algebra.
}
\begin{equation}
\sphMeas \left( \left\{ v \in \mathbb{S}^{d-1} \,\middle|\,
\| A v \| \le a \right\} \right) \le \frac{a^d}{|\det(A)|}.
\label{eq:measureBound}
\end{equation}
\label{le:measureBound}
\end{lemma}

\begin{figure}[b!]
\begin{center}
\setlength{\unitlength}{1cm}
\begin{picture}(8,6)
\put(0,0){\includegraphics[width=8cm]{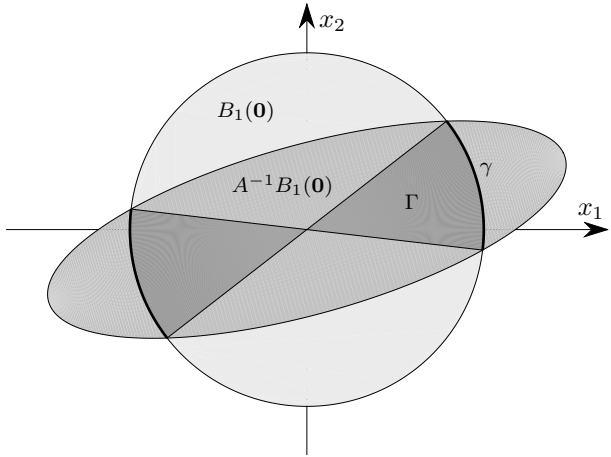}}
\put(7.6,3.2){\footnotesize $x_1$}
\put(4.16,5.7){\footnotesize $x_2$}
\put(2.8,4.5){\scriptsize $B_1(\b0)$}
\put(3,3.5){\scriptsize $A^{-1} B_1(\b0)$}
\put(5.3,3.28){\scriptsize $\Gamma$}
\put(6.3,3.76){\scriptsize $\gamma$}
\end{picture}
\caption{
A sketch of the geometric elements introduced in the proof of Lemma \ref{le:measureBound}.
\label{fig:schemPWExpLemma}
} 
\end{center}
\end{figure}

\begin{proof}
In view of the substitution $A \mapsto c A$ it suffices to consider $c = 1$.
The set under consideration is then
$\gamma = \left\{ v \in \mathbb{S}^{d-1} \,\middle|\, \| A v \| \le 1 \right\}$.
The left hand-side of \eqref{eq:measureBound} is
\begin{equation}
\sphMeas(\gamma) = \frac{{\rm meas}(\Gamma)}{{\rm meas}(B_1(0))},
\label{eq:measureBoundProof1}
\end{equation}
where $\Gamma = \left\{ \alpha v \,\middle|\, v \in \gamma,\, 0 \le \alpha \le 1 \right\}$.
Notice $\| A u \| \le 1$ for all $u \in \Gamma$.
Therefore $\Gamma$ is a subset of
$A^{-1} B_1(\b0) = \left\{ A^{-1} v \,\middle|\, v \in B_1(\b0) \right\}$, see Fig.~\ref{fig:schemPWExpLemma}.
Under multiplication by $A^{-1}$ the Lebesgue measure of a measurable set
is scaled by the factor $|\det(A^{-1})| = \frac{1}{|\det(A)|}$, thus
\begin{equation}
{\rm meas} \left( A^{-1} B_1(0) \right) = \frac{{\rm meas}(B_1(0))}{|\det(A)|}.
\label{eq:measureBoundProof2}
\end{equation}
Since $\Gamma \subset A^{-1} B_1(0)$,
by combining \eqref{eq:measureBoundProof1} and \eqref{eq:measureBoundProof2}
we obtain \eqref{eq:measureBound} with $c = 1$, as required.
\end{proof}

\begin{proof}[Proof of Theorem \ref{th:main}]
Choose any $\ee > 0$.
Let
\begin{equation}
\gamma_\ee = \left\{ v \in \mathbb{S}^{d-1} \,\middle|\,
\liminf_{n \to \infty} \frac{1}{n} \ln \left( \left\| C_n(x,v) v \right\| \right) \le \lambda_{\rm bound} - 2 \ee \right\},
\label{eq:gammaee}
\end{equation}
and
\begin{equation}
\gamma_{\ee,n} = \left\{ v \in \mathbb{S}^{d-1} \,\middle|\,
\frac{1}{n} \ln \left( \left\| C_n(x,v) v \right\| \right) \le \lambda_{\rm bound} - \ee \right\},
\label{eq:gammaeen}
\end{equation}
for all $n \ge 1$.
If $v \in \gamma_\ee$, then $v \in \gamma_{\ee,n}$, for infinitely many values of $n \ge 1$.
Thus for all $N \ge 1$ we have $\gamma_\ee \subset \bigcup_{n \ge N} \gamma_{\ee,n}$.
Below we show that
\begin{equation}
\sphMeas \left( \gamma_{\ee,n} \right) \le {\rm e}^{-\frac{d \ee n}{3}},
\label{eq:mainProof1}
\end{equation}
for sufficiently large values of $n$.
This will complete the proof because it implies \linebreak               
$\sphMeas \left( \bigcup_{n \ge N} \gamma_{\ee,n} \right)
\le \frac{{\rm e}^{\frac{-d \ee N}{3}}}{1 - {\rm e}^{\frac{-d \ee}{3}}} \to 0$ as $N \to \infty$.
Hence $\sphMeas \left( \gamma_\ee \right) = 0$
and so, because we can take $\ee > 0$ arbitrarily small,
\eqref{eq:mainBound} holds for almost all $v \in \mathbb{S}^{d-1}$.
The left hand-side of \eqref{eq:mainBound} is independent of $\| v \|$,
thus \eqref{eq:mainBound} also holds for almost all $v \in T \mathbb{R}^d$.

Given $n \ge 1$, let $\{ L, R \}^n$ denote the set of words of length $n$ involving the symbols $L$ and $R$.
We index the elements of any $\cS \in \{ L,R \}^n$ from $i = 0$ to $i = n-1$,
and write $\cS = \cS_0 \cS_1 \cdots \cS_{n-1}$.
Let
\begin{align}
\Xi_n = \big\{ \cS \in \{ L,R \}^n \,\big|\,
&\cS_i = L {\rm ~if~} f^i(x)_1 < 0 {\rm ~and~} \cS_i = R {\rm ~if~} f^i(x)_1 > 0, \nonumber \\
&{\rm for~all~} i \in \{ 0,\ldots,n-1 \} \big\}.
\label{eq:Xin}
\end{align}
The set $\Xi_n$ contains $2^{n - (\ell_n + r_n)}$ words because
to create a word $\cS \in \Xi_n$,
we are free to choose either $\cS_i = L$ or $\cS_i = R$ only if $f^i(x) = 0$.
The number of indices $i \in \{ 0,\ldots,n-1 \}$ for which
$f^i(x) = 0$ is $n - (\ell_n + r_n)$.
Thus we have $n - (\ell_n + r_n)$ independent choices between two symbols,
so a total of $2^{n - (\ell_n + r_n)}$ words.

Given $v \in \mathbb{S}^{d-1}$, by \eqref{eq:C} and \eqref{eq:Cn} we have
$C_n(x,v) = D f_{\cS_{n-1}} \left( f^{n-1}(x) \right) \cdots D f_{\cS_0}(x)$
for some $\cS \in \{ L,R \}^n$.
We can write this as
\begin{equation}
C_n(x,v) = D f_\cS(x),
\nonumber
\end{equation}
where $f_\cS = f_{\cS_{n-1}} \circ \cdots \circ f_{\cS_0}$
denotes the composition of $f_L$ and $f_R$ in the order determined by $\cS$.
But $\cS \in \Xi_n$ because by \eqref{eq:C} and \eqref{eq:Cn} we must have
$\cS_i = L$ if $f^i(x)_1 < 0$ and $\cS_i = R$ if $f^i(x)_1 > 0$,
for all $i \in \{ 0,\ldots,n-1 \}$.

We cannot apply Lemma \ref{le:measureBound} to the set $\gamma_{\ee,n}$
by using $A = C_n(x,v)$ because this matrix depends on $v$.
For this reason we introduce the set
\begin{equation}
\gamma_{\ee,n}^\cS = \left\{ v \in \mathbb{S}^{d-1} \,\middle|\,
\frac{1}{n} \ln \left( \left\| D f_\cS(x) v \right\| \right) \le \lambda_{\rm bound} - \ee \right\},
\nonumber
\end{equation}
for a given word $\cS \in \Xi_n$.
Then $\gamma_{\ee,n} \subset \bigcup_{\cS \in \Xi_n} \gamma_{\ee,n}^\cS$, and so
\begin{equation}
\sphMeas(\gamma_{\ee,n}) \le \sum_{\cS \in \Xi_n} \sphMeas \left( \gamma_{\ee,n}^\cS \right).
\label{eq:mainProof2}
\end{equation}
By Lemma \ref{le:measureBound}\removableFootnote{
The $\ee$'s have been introduced in view of this step.
For example if $f$ is the piecewise-linear map $g$ with $A_L = A_R = k I$, where $k > 1$,
then the right hand-side of \eqref{eq:mainProof3} is $\frac{k^n {\rm e}^{-d n \ee}}{k^n} = {\rm e}^{-d n \ee}$.
For this to go to zero as $n \to \infty$ exponentially (or even just be less than $1$)
we need $\ee > 0$.
},
\begin{equation}
\sphMeas \left( \gamma_{\ee,n}^\cS \right) \le
\frac{{\rm e}^{d n (\lambda_{\rm bound} - \ee)}}{\left| \det \left( D f_\cS(x) \right) \right|}.
\label{eq:mainProof3}
\end{equation}

For the remainder of the proof we assume $a_L \ge a_R$ without loss of generality.
By \eqref{eq:aLaR}, $\left| \det \left( D f_\cS(x) \right) \right| \ge a_L^k a_R^{n-k}$,
where $k$ is the number of $L$'s in the word $\cS$.
Since $\ell_n \le k \le n - r_n$ and $a_L \ge a_R$ we have
\begin{equation}
\left| \det \left( D f_\cS(x) \right) \right| \ge a_L^{\ell_n} a_R^{n-\ell_n}.
\label{eq:mainProof4}
\end{equation}
By substituting \eqref{eq:lambdabound} and \eqref{eq:mainProof4} into \eqref{eq:mainProof3} we obtain
\begin{equation}
\sphMeas \left( \gamma_{\ee,n}^\cS \right) \le
\frac{2^{(\ell + r - 1)n} c^n {\rm e}^{-d n \ee}}{c_L^{\ell_n} c_R^{n-\ell_n}}.
\nonumber
\end{equation}
Since $\Xi_n$ contains $2^{n - (\ell_n + r_n)}$ words, \eqref{eq:mainProof2} then implies
\begin{equation}
\sphMeas(\gamma_{\ee,n})
\le \frac{2^{(\ell + r)n - (\ell_n + r_n)} c^n {\rm e}^{-d n \ee}}{c_L^{\ell_n} c_R^{n-\ell_n}}
= \left( \frac{2^{\ell + r - \frac{\ell_n + r_n}{n}} c \,{\rm e}^{-d \ee}}
{c_L^{\frac{\ell_n}{n}} c_R^{1 - \frac{\ell_n}{n}}} \right)^n.
\label{eq:mainProof5}
\end{equation}

Since $\liminf_{n \to \infty} \frac{\ell_n + r_n}{n} = \ell + r$,
there exists $N_1 \in \mathbb{Z}$ such that for all $n \ge N_1$ we have
$\frac{\ell_n + r_n}{n} \ge \ell + r - \frac{d \ee}{3 \ln(2)}$, that is
$2^{\ell + r - \frac{\ell_n + r_n}{n}} \le {\rm e}^{\frac{d \ee}{3}}$.
Similarly, since $\liminf_{n \to \infty} \frac{\ell_n}{n} = \ell$,
there exists $N_2 \in \mathbb{Z}$ such that for all $n \ge N_2$ we have
$\left( \frac{c_L}{c_R} \right)^{\frac{\ell_n}{n}} \ge
\left( \frac{c_L}{c_R} \right)^\ell {\rm e}^{\frac{-d \ee}{3}}$, that is
$\frac{c}{c_L^{\frac{\ell_n}{n}} c_R^{1 - \frac{\ell_n}{n}}} \le {\rm e}^{\frac{d \ee}{3}}$.
By inserting these bounds into \eqref{eq:mainProof5} we obtain \eqref{eq:mainProof1}
(valid for all $n \ge \max(N_1,N_2)$) as required.
\end{proof}

\subsection{Numerical simulations}
\label{sub:numerics}

\begin{figure}[b!]
\begin{center}
\setlength{\unitlength}{1cm}
\begin{picture}(12,9.2)
\put(0,0){\includegraphics[height=9cm]{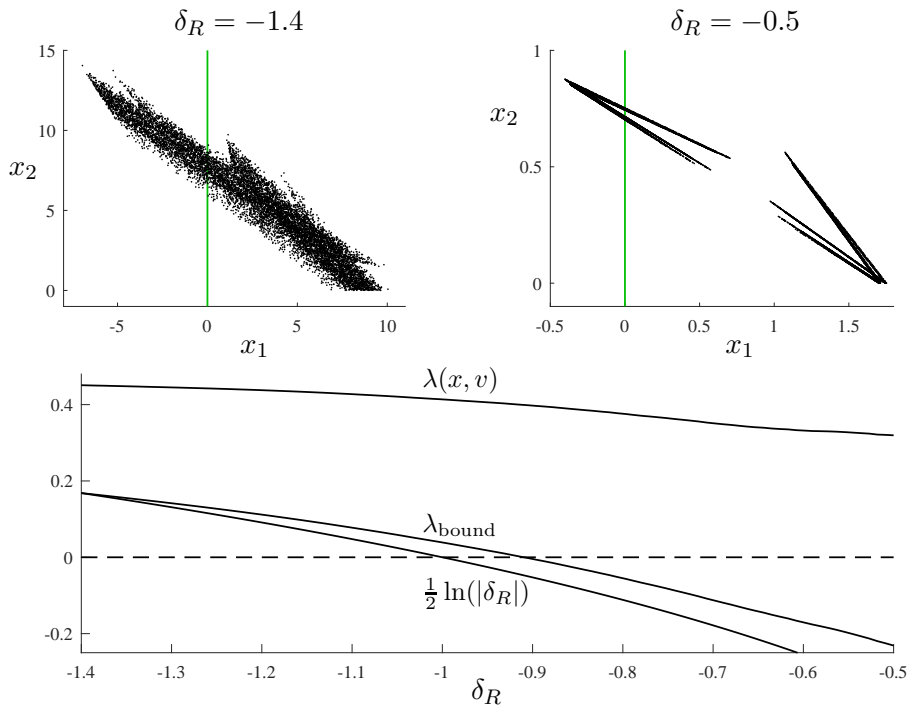}}
\put(6.15,0){\small $\delta_R$}
\put(5.5,1.32){\footnotesize $\frac{1}{2} \ln(|\delta_R|)$}
\put(5.5,2.22){\footnotesize $\lambda_{\rm bound}$}
\put(5.5,4.16){\footnotesize $\lambda(x,v)$}
\put(3.08,4.61){\small $x_1$}
\put(0,6.98){\small $x_2$}
\put(9.53,4.61){\small $x_1$}
\put(6.4,7.7){\small $x_2$}
\put(2.2,8.9){\small $\delta_R = -1.4$}
\put(8.8,8.9){\small $\delta_R = -0.5$}
\end{picture}
\caption{
The lower plot shows the maximal Lyapunov exponent and two lower bounds for the
two-dimensional border-collision normal form \eqref{eq:bcnf}
with \eqref{eq:params} and $\mu = 1$.
The upper plots are phase portraits showing the numerically computed attractor.
\label{fig:numericsPWExpanding_a}
} 
\end{center}
\end{figure}

\begin{figure}[b!]
\begin{center}
\setlength{\unitlength}{1cm}
\begin{picture}(12,9.2)
\put(0,0){\includegraphics[height=9cm]{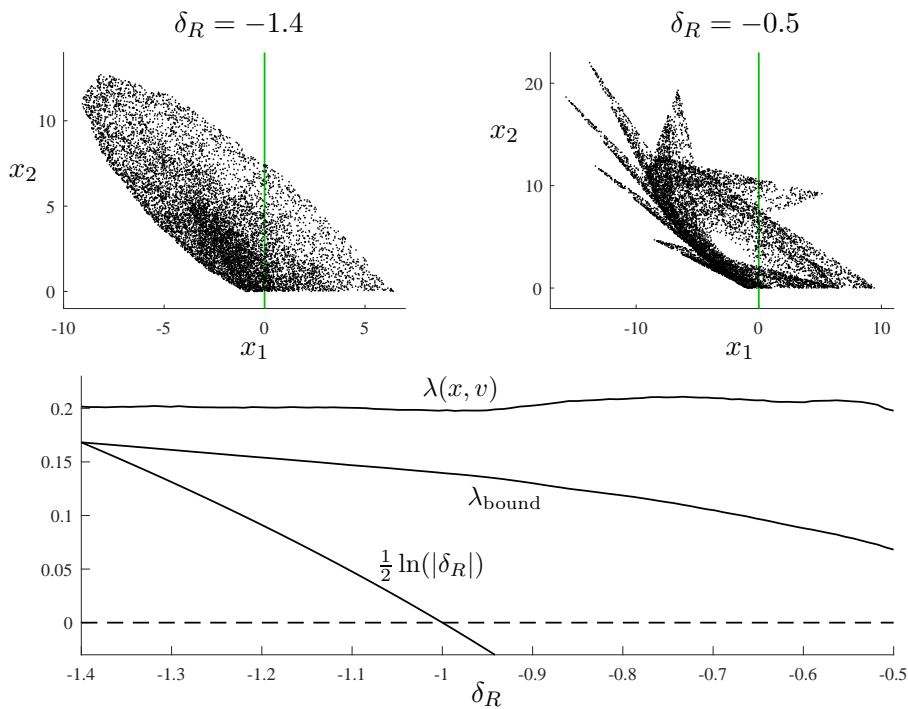}}
\put(6.15,0){\small $\delta_R$}
\put(4.9,1.72){\footnotesize $\frac{1}{2} \ln(|\delta_R|)$}
\put(6.1,2.61){\footnotesize $\lambda_{\rm bound}$}
\put(5.5,4.06){\footnotesize $\lambda(x,v)$}
\put(3.08,4.61){\small $x_1$}
\put(0,6.98){\small $x_2$}
\put(9.53,4.61){\small $x_1$}
\put(6.4,7.5){\small $x_2$}
\put(2.2,8.9){\small $\delta_R = -1.4$}
\put(8.8,8.9){\small $\delta_R = -0.5$}
\end{picture}
\caption{
This repeats Fig.~\ref{fig:numericsPWExpanding_a} for $\mu = -1$.
\label{fig:numericsPWExpanding_b}
} 
\end{center}
\end{figure}

Here we illustrate Theorem \ref{th:main} with the two-dimensional border-collision normal form \eqref{eq:bcnf}.
We found in \S\ref{sub:OmegaNumerically} that with \eqref{eq:params} and $\mu = 0$
the fixed point $x = \b0$ appears to be asymptotically stable for all $-1.46 \le \delta_R \le -0.41$.
Then Theorem \ref{th:attractingSet} implies that the map has an attractor near $x = \b0$ 
for sufficiently small $|\mu|$.
But \eqref{eq:bcnf} is piecewise-linear---the structure of the dynamics
is independent of the magnitude of $\mu$---hence
there exists a bounded attractor for all $\mu \in \mathbb{R}$.

Numerical investigations suggest that this attractor is unique.
For the two particular values $\delta_R = -1.4$ and $\delta_R = -0.5$
the attractor is shown in Fig.~\ref{fig:numericsPWExpanding_a} for $\mu = 1$
and Fig.~\ref{fig:numericsPWExpanding_b} for $\mu = -1$.
These figures also show numerically computed values for
the Lyapunov exponent $\lambda(x,v)$ \eqref{eq:lyapPWS}
and the lower bound $\lambda_{\rm bound}$ \eqref{eq:lambdabound}.
For each value of $\delta_R$ these were computed
from $10^6$ iterates of the forward orbit of $x = \b0$ with the first $100$ (transient) iterates removed.
For the computation of $\lambda(x,v)$ we used $v = \begin{bmatrix} 1 \\ 0 \end{bmatrix}$.

By Theorem \ref{th:main} we expect $\lambda(x,v) > \lambda_{\rm bound}$,
and this is indeed the case.
Indeed $\lambda(x,v) > 0$ for all values of $\delta_R$ in
Figs.~\ref{fig:numericsPWExpanding_a} and \ref{fig:numericsPWExpanding_b}
suggesting the map has a chaotic attractor for all such $\delta_R$ and all $\mu \ne 0$.

As discussed in Remark \ref{re:bound3} we can construct a simpler bound
that does not require knowledge of the forward orbit of $x$.
Assuming $\ell + r = 1$ we have
$\lambda_{\rm bound} \ge \frac{1}{d} \ln(\min(a_L,a_R))$.
Since \eqref{eq:bcnf} is piecewise-linear, $a_L = |\delta_L|$ and $a_R = |\delta_R|$.
Here $|\delta_L| \ge |\delta_R|$ thus
$\lambda_{\rm bound}$ is bounded by $\frac{1}{d} \ln(|\delta_R|)$
which we have also plotted.

\section{Discussion}
\label{sec:conc}
\setcounter{equation}{0}

Chaotic attractors of piecewise-smooth maps are useful in cryptography \cite{KoLi11}
but undesirable in most engineering and control applications \cite{ZhMo03}.
In both settings it is helpful to understand parameter regions where chaotic attractors exist.

In this paper we have shown how the existence of a topological attractor
follows from the asymptotic stability of a fixed point on a switching manifold.
It is well known that stability can often be established by constructing a Lyapunov function.
In particular there are well-established methods by which
the existence of a piecewise-quadratic Lyapunov function can be verified \cite{Jo03,LiAn09}.
However, these methods fail in some instances for which the fixed point is stable
because only a limited class of Lyapunov functions is considered.

For this reason, following \cite{AtLa14}, here we advocate
condition \eqref{it:3} of Theorem \ref{th:stability} for demonstrating stability.
This condition characterises asymptotic stability exactly and, as discussed in \S\ref{sub:OmegaNumerically},
is readily amenable to an accurate and efficient numerical implementation for piecewise-linear maps.
We achieved this here for the two-dimensional border-collision normal form with $\mu = 0$
and showed that the origin can be stable even if both pieces of the map are area-expanding.
This is possible because the map is non-invertible over the given parameter range
and the expansion competes with the contractive effect of folding at the switching manifold.

In \S\ref{sub:bound} we obtained the lower bound \eqref{eq:lambdabound} on the maximal Lyapunov exponent.
For the two-dimensional border-collision normal form with $\mu \ne 0$,
if $\nu(\Sigma) = 0$ (where $\nu$ is the invariant probability measure of an attractor),
which appears to be the case for the four numerically computed attractors shown in
Figs.~\ref{fig:numericsPWExpanding_a} and \ref{fig:numericsPWExpanding_b},
then we expect to have $\lambda_{\rm bound} \ge \frac{1}{2} \ln \big( {\rm min}(|\delta_L|,|\delta_R|) \big)$.
This immediately gives $\lambda_{\rm bound} > 0$ in the area-expanding case
and thus chaos in the sense of a positive Lyapunov exponent.
The necessity of the assumption $\nu(\Sigma) = 0$ is seen by simply putting $\mu = 0$.
In this case $\nu(\Sigma) = 1$, where $\nu$ is the Dirac measure corresponding the fixed point $x = \b0$,
and this point may be stable.

In order to reveal the full power of Theorem \ref{th:main}
it remains to apply bounds on the values of $\ell$ and $r$
obtained from restrictions to the possible symbolic dynamics and
apply the bound \eqref{eq:lambdabound}
to attractors for which $0 < \nu(\Sigma) < 1$,
but examples of this are not known for the border-collision normal form.
It also remains to obtain tighter bounds on maximal Lyapunov exponent
by using more information about $A_L$ and $A_R$ than simply their determinants.

\section*{Acknowledgements}

This work was supported by Marsden Fund contract MAU1809, managed by Royal Society Te Ap\={a}rangi.


\end{document}